\documentclass [11pt,reqno]{amsart}
\usepackage {amsmath,amssymb,verbatim,geometry}
\usepackage[all]{xy}
\newif\ifpdf
\ifpdf
\usepackage[pdftex]{hyperref}
\else
\fi

\numberwithin{equation}{section}       

\newtheorem{prop} {Proposition} [section]
\newtheorem{thm}[prop] {Theorem} 
 
\newtheorem{defi}[prop] {Definition}
\newtheorem{lem}[prop] {Lemma}

\newtheorem{prop-def}[prop]{Proposition-Definition}

\newtheorem{rem}[prop]{Remark}

\newcommand{\C}{{\mathbb{C}^n}}

\newcommand{\bigzero}{\mbox{\normalfont\Large 0}}

\newcommand{\tr}{\operatorname{tr}}

\newcommand{\Om}{{\Omega}}
\newcommand{\pa}{{\partial}}
\newcommand{\Le}{{\mathcal L}}

\title{A remark on Oka's lemma and a geometric property of pseudoconvex domains}

\date{\today}

\author{ S.Dinew*, \.Z. Dinew }

\address{Jagiellonian University 30-348 Krak{\'o}w, {\L}ojasiewicza 6, Poland}
\email{slawomir.dinew@im.uj.edu.pl}

\address{Jagiellonian University 30-348 Krak{\'o}w, {\L}ojasiewicza 6, Poland}
\email{zywomir.dinew@uj.edu.pl}

\subjclass[2020]{Primary: 32U30; Secondary  31B05, 32D20, 32U15, 32E40, 32T27, 51K05, 53C15\\
	* This author is partially supported  by grant no. 2021/41/B/ST1/01632
	from the National Science Center, Poland}

\begin{document}

	\begin{abstract}
	A direct proof of Oka's lemma on the relation of holomorphic convexity and the properties of the distance to the boundary function  is provided.	Some related problems for subharmonicity properties of this function are also studied. A new geometric property of pseudoconvex domains is described.
	\end{abstract}
\maketitle

	\section*{Introduction}
A textbook result, due to K. Oka, states that a domain of holomorphy $\Om$ in $\C$, not identical with the whole space, is (Hartogs) pseudoconvex i.e., the function $-\log(d_{\pa\Om}(z))$ is plurisubharmonic in $\Om$ with $d_{\pa\Om}$ being the Euclidean distance to $\pa\Om$.

The classical proof of Oka's lemma (see for example \cite{GR}, Chapter IX, Theorem D.4) is by contradiction. More precisely, the non-plurisubharmonicity of $-\log(d_{\pa\Om}(z))$ yields the existence of a pathological family of holomorphic discs near a boundary point of $\Om$ which leads to a contradiction. In \cite{HM} a question was raised whether a {\it direct proof} is possible. The main thrust would then be to show the result for smoothly bounded domains of holomorphy, which are well known to be {\it Levi pseudoconvex}. In this direction the following result was established:
\begin{thm}[Herbig-McNeal, \cite{HM}]\label{smooth}
Let $\Om$ be a smoothly bounded Levi pseudoconvex domain in $\C$. Then, there is a neighborhood $U$ of $\pa\Om$, such that $-\log(d_{\pa\Om}(z))$ is plurisubharmonic on $\Om\cap U$.	
\end{thm}
Of course, the main interest towards Oka's lemma, just as in Theorem \ref{smooth}, is close to the boundary as the {\it signed distance function}
\begin{equation}\label{sdf}
	\delta(z)=\begin{cases}-d_{\pa\Om}(z),\ \ z\in\Om;\\
		d_{\pa\Om}(z),\ \ z\in\C\setminus\Om 
		\end{cases}
\end{equation}
is, still for a smoothly bounded domain $\Om$, a {\it defining function} with excellent analytic properties. It is nevertheless interesting to obtain a direct proof covering the plurisubharmonicity in the {\it whole} $\Om$ rather than just close to the boundary. It should be pointed out that, contrary to \cite{HM}, non-smooth analysis has to be applied in such a project as the function $-\log(d_{\pa\Om}(z))$ is never smooth for a bounded $\Om$.

The aim of this note is to provide a direct proof of the full version of Oka's lemma, that is the statement:

\begin{thm}[Oka]\label{oka}
	Let $\Om\subsetneq\C$ be a domain. Then, $\Omega$ is pseudoconvex if and only if the function $-\log(d_{\pa\Om}(z))$ is plurisubharmonic on the whole domain $\Om$.	
\end{thm}

This extends the result of Herbig and McNeal. In fact, we show that the result is a simple consequence of the modern development of the theory of the distance function, suitably coupled with a recent removal of singularities theorem shown in \cite{DD}. Additionally, we hope to advertise the mentioned developments which do not seem to be well-known in the complex-analytic community.

The following result is undoubtedly known, yet we could not find a reference in the literature. We provide a proof, parallel to the proof of the main theorem, just to illustrate the method.
\begin{thm}\label{sh} Let $\Omega\neq\mathbb R^{m}$ be any domain in $\mathbb R^{m}$. Then,  the function $-\log(d_{\pa\Om}(x))$ if $m=1,2$ and $(d_{\pa\Om}(x))^{2-m}$ if $m>2$ is subharmonic on $\Omega$. Moreover, if $\Omega$ is smoothly bounded and $m>2$ then $-\log(d_{\pa\Om}(x))$ is subharmonic in $\Omega$ if the mean curvature  $H$ of the boundary satisfies  $H\geq \frac{-1}{(m-2) R}$ where $R$ is the inradius of $\Omega$, that is, the radius of the largest ball contained in $\Omega$. The function $-\log(d_{\pa\Om}(x))$ is subharmonic near the boundary for any smoothly bounded $\Omega$. 
\end{thm}
For $m=1$ this boils down to the trivial fact that for $a<x<b$ the function $\max\{-\log(x-a),-\log(b-x)\}$ is convex. For $m=2$ this result reflects the well-known fact that every planar domain is pseudoconvex.

For general domains $\Omega\subsetneq \mathbb R^{m},\, m>2$ the function $-\log(d_{\pa\Om}(x))$ may fail to be subharmonic in the whole $\Omega$. For example, if $r$ is small (it is enough to assume $0<r<\frac{m-2}{m}$) and $\Omega$ is the annular region
$$\Omega:=\{ x\in \mathbb R^{m} |\quad r<\Vert x\Vert <1\}=B(0,1)\setminus \overline{B(0,r)}$$ then $d_{\pa\Om}(x)=\Vert x\Vert -r$ for all $x$ such that  $r<\Vert x\Vert<\frac{1+r}{2}$. One can check that $$\Delta (-\log(\Vert x\Vert -r))=\frac{(m-1)r-(m-2)\Vert x\Vert}{\Vert x\Vert (\Vert x\Vert -r)^2}$$ which is negative for all $x$ such that $\frac{m-1}{m-2}r<\Vert x\Vert<\frac{1+r}{2}$. In particular, as $B(0,1)\setminus \overline{B(0,r)}= \{x \quad |\quad \Vert x \Vert ^d + (1+r-\Vert x \Vert)^d< 1+r^d\}\text{ for any } d<2-m$, the subharmonicity may fail for domains which are smooth sublevel sets of smooth strongly subharmonic functions, see also Remark \ref{shsublevel}.


For pseudoconvex domains the inequality $H\geq \frac{-1}{(m-2) R}$ may also fail (again, the planar domain $B(0,1)\setminus\overline{B(0,r)}$ is an example) but  $-\log(d_{\pa\Om}(x))$ is still subharmonic as a consequence of Oka's lemma. We provide a direct proof of this, not appealing to the plurisubharmonicity of $-\log(d_{\pa\Om}(x))$. This follows from a geometric property of pseudoconvex domains which is of independent interest and which, we believe, has not been observed so far:
\begin{thm}\label{geometric property} Let $\Om\subset\C$ be a bounded pseudoconvex domain with $C^2$ smooth boundary. Then at any point $w\in\pa\Omega$ one can choose $2n-2$ out of the $2n-1$ principal curvatures of $\pa\Om$ so that their sum
	$$\kappa_1+\cdots+\kappa_{j-1}+\kappa_{j+1}+\cdots+\kappa_{2n-1}$$
	is non-negative. The sum is positive unless $\pa \Om$ is Levi flat at $w$.  Moreover, at least $n-1$ of the principal curvatures are non-negative (positive if $\pa\Om$ is strongly pseudoconvex at $w$).
\end{thm}
We remark that on smoothly bounded pseudoconvex domains it is not possible to bound from below and hence to control the sum of all of the $2n-1$ principal curvatures, nor a particular principal curvature, as demonstrated by the above example.

It is known that a stronger assumption than in Theorem \ref{sh}, namely that $-d_{\pa\Om}(x)$ is subharmonic outside a prescribed compact singular subset of $\Omega$, called the cut locus $\Sigma$ below, is equivalent to the so-called {\it mean convexity} of a smoothly bounded domain $\Omega$, see \cite{LLL}. This equivalence is utilized in \cite{LLL} to characterize the domains with $C^2$ boundary on which the Hardy's inequality holds with optimal constant. The mean convexity of a domain is the property that the mean curvature of the boundary $\partial \Omega$, considered as a hypersurface, is non-negative at each boundary point. For more on the notion of mean convexity, its significance in geometry and applications see \cite{BEL,Ma} and for a introductory level exposition see \cite{G}.

  We observe that the same argument as in the proof of our main theorem allows one to restate the equivalence theorem of \cite{LLL} in a more transparent way:
  \begin{thm}\label{meanconvex} A smoothly bounded domain $\Omega\subset \mathbb R^{m}$ is mean convex if and only if $-d_{\pa\Om}(x)$ is subharmonic throughout $\Omega$.
  	\end{thm}
  
 It follows that, by using Theorem \ref{meanconvex}, one can meaningfully define the mean convexity of not necessarily smooth domains by the property that $-d_{\pa\Om}(x)$ is subharmonic or (which is equivalent) that $\Omega$ may be exhausted by mean convex domains with smooth boundary. 
 
{\bf Acknowledgment} The first author is indebted to professor Franc Forstneri\v{c} for helpful discussions on \cite{Fo}.
	\section{Preliminaries}
	
	In this section we gather the notions and basic results used in our proof of Theorem \ref{oka}.
	
	{\bf 1. Complex analytic notions.} Below we recall the main notions related to complex convexity that will be utilized later on. For more details we refer to \cite{GR} or \cite{K}.
	\begin{defi} A domain $\Om\subset\C$ is said to be a {\bf domain of holomorphy} 
		if the following holds: there are no open sets $\Om_1,\Om_2$ such that $\Om_2$ is connected, $\Omega_2\not\subset\Omega$, $\varnothing \neq\Om_1\subset\Om_2\cap\Om$ and for any holomorphic function $u$ on $\Om$ there is a holomorphic function $u_2$ on $\Om_2$, such that $u=u_2$ on $\Om_1$.
		\end{defi}
	
	It is a classical result in complex analysis (see \cite{K}) that the above definition is equivalent to {\it holomorphic convexity}:
	
		\begin{defi} A domain $\Om\subset\C$ is said to be  {\bf holomorphically convex} if for any compact set $K\subset\Om$ its holomorphic hull
		$$\hat{K}:=\left\lbrace z\in \Om\ |\ \forall f\in\mathcal O(\Om)\ |f(z)|\leq \sup_K|f|\right \rbrace$$
		is also compact in $\Om$.
	\end{defi}

	If the boundary of $\Om$ is at least $C^2$ smooth the notion of Levi pseudoconvexity can be introduced:
	\begin{defi}
	A smoothly bounded domain $\Om\subset \C$ is {\bf Levi pseudoconvex} if for any $p\in\pa\Om$ and any local defining function $\rho$ of $\pa\Om$ near $p$ the inequality
	$$\Le(\rho(p))(X,X):=\sum_{j,k=1}^n\frac{\pa^2\rho}{\pa z_j\pa\bar{z}_k}(p)X_j\overline{X}_k\geq 0$$  
	holds for any complex tangent vector $X=\sum_{j=1}^nX_j\frac{\pa}{\pa z_j}\in T_{p}^{\mathbb C}\pa\Om$. The expression $\Le(\rho(p))(X,X)$ is the {\bf Levi form} of $\rho$ at the boundary point $p$ evaluated on the vector $X$.
	\end{defi}
It is a simple exercise to check that the inequality in the definition above does not depend on the choice of the local defining function. It is also invariant under a holomorphic change of variables near $p$.

A classical result in complex analysis (see \cite{GR}, Chapter IX, Section B) is that domains of holomorphy in $\C$ with $C^2$ smooth boundaries are Levi pseudoconvex.

{\bf 2. The distance function.} The distance function and its signed version are classical objects of study. Unfortunately, the basic facts of the associated theory are scattered in the literature - we refer to \cite{CM, E, Fe, Ha, KS, KP, LN1, LN2, MM} for a (necessarily incomplete) list of old and new contributions. As the distance function does not depend on the complex structure, we identify $\C$ with $\mathbb R^{m}$ for $m=2n$ in this section and  recall the relevant definitions and facts in the real setting:
\begin{defi}
Let  $\Om\subset\mathbb R^{m}, \Om\neq\mathbb R^{m}$  be a domain. The set of points $x\in\Om$ such that there is more than one point on $\pa\Om$ realizing the distance from $x$ to $\pa\Om$ is called the {\bf medial axis} of $\pa\Om$. We denote this set by $\Gamma$.	
\end{defi}
As $d_{\pa\Om}(x)$ is $1$-Lipschitz and differentiable at $\Om\setminus\Gamma$ (see \cite{Fe}, at such points one has $\nabla d_{\pa\Om}(x)=\frac{x-w}{\Vert x-w\Vert}$ with $w$ being the unique closest boundary point) it follows that $\Gamma$ is of Lebesgue measure zero. Much more is in fact true - P. Erd\"os in \cite{E} proved that $\Gamma$ is $(m-1)$-{\it rectifiable} (see also \cite{Ha} for a modern treatment). Topologically, $\Gamma$ need not be a closed set as the example of a planar ellipse shows.

\begin{defi}
The Euclidean closure of $\Gamma$ in $\Om$ is called the {\bf cut locus}. We denote the cut locus by $\Sigma$.
\end{defi}

A beautiful example of Mantegazza and Mennucci (see \cite{MM}) shows that $\Sigma$ can have positive Lebesgue measure even if $\pa\Om$ is $C^{1,1}$ smooth. The crucial fact we shall need is that this does not happen if we add a little to the boundary regularity:

\begin{lem}[Crasta-Malusa, \cite{CM}]\label{CrMa}
Let  $\Om\subset \mathbb R^{m}$ be a bounded domain with $C^2$-smooth boundary. Then $\Sigma$ is of zero Lebesgue measure. 
\end{lem}

\begin{rem}
	If $C^3$ boundary regularity is assumed, Nirenberg and Li have proved that the $(m-1)$-dimensional Hausdorff measure of $\Sigma$ is finite, which is optimal (see \cite{LN2}). In our proof we can use either result.
\end{rem}
Under the $C^2$ boundary assumption, $\Sigma$ is also known to be a strong deformation retract of $\Omega$, see \cite{W}, which yields that $\Sigma$ encodes many topological properties of $\Omega$.

Our main interest towards $\Gamma$ and $\Sigma$ lies in the regularity properties of $d_{\pa\Om}$ in their complement. The following result is classical.

\begin{lem}[\cite{Fe},\cite{KP},\cite{LN1}] If $\pa \Om$ is $C^{1,1}$ smooth and bounded then $\Gamma$ is relatively compact in $\Om$.
	\end{lem}
Geometrically, the lemma above follows from the fact that such domains satisfy the {\it interior ball condition}. In particular, there is a neighborhood $U$ of $\pa\Om$ (the very same neighborhood used in Theorem \ref{smooth}) where $d_{\pa\Om}(z)$ is differentiable.

When it comes to the set $\Om\setminus U$ it remains to observe that $d_{\pa\Om}$ is differentiable on $\Om\setminus\Gamma$ and continuously differentiable on $\Om\setminus\Sigma$  as there (and not only on $\Omega\cap U$) each point has a unique closest point on $\pa\Om$. Further regularity of $d_{\pa\Om}$ follows from the classical lemma of Gilbarg and Trudinger \cite{GT}, see also \cite{KP}. We formulate the relevant result in a way suitable for our needs. The calculus can be found in Gilbarg-Trudinger, (\cite{GT}, Lemma 14.17) for points near the boundary and in \cite{LLL} in full generality.
\begin{lem}[  ]\label{GiTr}
Let $\Om\subset \mathbb R^{m}$ be a bounded domain with $C^k$-smooth boundary, $k\geq2$. Then $d_{\pa\Om}$ is $C^k$ smooth on $\overline{\Om}\setminus\Sigma$.	In local real orthogonal coordinates $(t_1,\cdots,t_{m})=(t',t_{m})$, called principal coordinates, such that $\frac{\pa}{\pa t_{m}}$ is the inward normal at the origin and near the origin $\pa\Om$ is given by $\pa\Om=(t',h(t'))$, $h\in C^k,\ h(0')=0, \nabla h(0')=0'$ and 
$$\frac{\pa^2 h}{\pa t_i\pa t_j}(0')=\kappa_i\delta_{ij},\quad i,j=1,\cdots,m-1$$
(here $\kappa_j$'s are the principal curvatures of $\pa\Om$ at the origin) the real Hessian of $d_{\pa\Om}$ at $(0',t_{m})$ reads
$$  \frac{\pa^2 d_{\pa\Om}}{\pa t_i\pa t_j}(0',t_{m})=\begin{pmatrix}
 { \frac{-\kappa_{i}}{1-t_{m}\kappa_{i}}\delta_{ij}}&\begin{matrix}0 \\ \vdots \end{matrix} \\ 

\begin{matrix}0 & \cdots \end{matrix} &0	
\end{pmatrix}_{i,j=1}^{m}.$$
\end{lem}

We remark that the formula holds along the ray spanned by $\frac{\pa}{\pa t_{m}}$ as long as $(0',0)$ remains the unique closest point to the boundary. Another remark is that the denominators $1-t_{m}\kappa_{i}$ remain positive whenever the formula holds and in particular provide a bound for the location of closest point from $\Sigma$ along the ray.

{\bf 3. Generalized Laplacians and removal of singularities.}
Recall the definition of the {\it generalized  upper Laplace parameter} (or the {\it upper Privalov operator} as it is sometimes called): 
\begin{defi}
Let $\Omega\subset \mathbb R^{m}$ be a domain and let $u\neq-\infty$ be an upper semicontinuous function on $\Omega$. For $x\in\Omega$ and $r$ such that the ball $\overline{B(x,r)}\subset \Omega$ one defines the integral mean of $u$ over the ball by  $$M_u(x,r):=\frac{\int_{B(x,r)} u d\lambda^{m} }{\lambda^{m}(B(x,r))}\in[-\infty,\infty),$$
(with $\lambda^{m}$ denoting the $m$-dimensional Lebesgue measure).

Then, the {\bf generalized  upper Laplace parameter} is defined by
$$\overline{\Delta} u(x):=2(m+2)\limsup_{r\to 0^+}\frac{M_u(x,r)-u(x)}{r^2}$$
for any $x\in\Omega\setminus\{u=-\infty\}$.
\end{defi}


For a $C^2$ function $\overline{\Delta} u(x)$ is the classical Laplacian of $u$ at $x$, hence $\overline{\Delta}$ is a kind of a generalized Laplacian. A modern exposition to these matters, with some recent improvements, is given in \cite{P, SS}. We will need two results about $\overline{\Delta}$ which can be found there. The first one relates the subharmonicity of $u$ to the values of the generalized  upper Laplace parameter.

\begin{thm}[Blaschke-Privalov criterion] An upper semicontinuous function $u\neq-\infty$ on $\Omega$ is subharmonic if and only if $\overline{\Delta} u(x)\geq 0$ for any $x\in\Omega\setminus\{u=-\infty\}$.
\end{thm}
The second one, which is a stronger result reads:
\begin{thm}[Privalov]
	An upper semicontinuous function $u\neq-\infty$ on $\Omega$ is subharmonic if and only if $\overline{\Delta} u(x)\geq 0$ almost everywhere in $\Omega$ with respect to the $m$ dimensional Lebesgue measure $\lambda^{m}$ and $\overline{\Delta} u(x)> -\infty$ for any $x\in\Omega\setminus(\{u=-\infty\}\cup P)$, where $P$ is a polar set.
\end{thm}

We will need the following simple observation:
\begin{lem}\label{supf} Let $\Omega\subset\mathbb R^{m}$ be a domain and let $\mathcal F$ be a family of upper semicontinuous functions on $\Omega$ which are locally uniformly bounded above and such that for any $\varphi\in \mathcal F$ one has $\overline{\Delta} \varphi\geq \psi$, where $\psi$ is a fixed lower semicontinuous function on $\Omega$. Then, the upper semicontinuous regularization of the supremum of the family, defined as
	$$\phi^{\ast}(x):=\limsup_{\Omega \ni y\to x} \phi (y),\quad \text{ where } \phi (y):=\sup_{\varphi\in\mathcal F} \varphi (y) $$
	also satisfies
	$$\overline{\Delta} \phi^{\ast}\geq \psi.$$
\end{lem}
\begin{proof} 
	
	Fix $x\in\Omega$ and $\varepsilon>0$ such that $\overline{B(x,\varepsilon)}\subset\Omega$. Let $C_{\varepsilon}:=\min_{y\in \overline{B(x,\varepsilon)}}\psi(y)$.  Take any $\varphi\in\mathcal F$. As $\overline{\Delta}\left( \varphi(y)-\frac{C_{\varepsilon}}{2}\Vert y\Vert^2\right)\geq 0$ throughout $B(x,\varepsilon)$, the functions $\varphi(y)-\frac{C_{\varepsilon}}{2}\Vert y\Vert^2$ are subharmonic in $B(x,\varepsilon)$, by the theorem of Blaschke-Privalov. But then $\phi^{\ast}(y)-\frac{C_{\varepsilon}}{2}\Vert y\Vert^2$ is also subharmonic and again by the Blaschke-Privalov theorem one has  $\overline{\Delta} \phi^{\ast}\geq C_{\varepsilon}$ in $B(x,\varepsilon)$. Finally,
	$$\overline{\Delta} \phi^{\ast}(x)\geq \lim_{\varepsilon\to 0^{+}} C_{\varepsilon}=\lim_{\varepsilon\to 0^{+}}\min_{y\in \overline{B(x,\varepsilon)}}\psi(y)=\psi(x),$$ by the lower semicontinuity of $\psi$.
\end{proof}

The following theorem from \cite{DD} will be crucial:
\begin{thm}\label{didi}
	Let $\Om\subset \C$ be a domain and $E\subset\Om$ be a closed set of zero Lebesgue measure. Let $u$ be a subharmonic function on $\Om$ which is plurisubharmonic on $\Om\setminus E$. Then $u$ is plurisubharmonic in the whole domain $\Om$. 
\end{thm}

{\bf 4. Principal curvatures and the mean curvature.}
  The notions and formulas in this part are standard in differential geometry, although in the vast majority of the references they are presented only for surfaces in $\mathbb R^3$. We feel, however, that for people from other areas it is far from obvious why the mean curvature is given by an expression as in \eqref{meancurv} or \eqref{implicitform} below. For the benefit of the reader we provide the details. The computations are also helpful to understand the geometric background of the problem.
    
   For a general $C^2$ hypersurface $M$ in $\mathbb R^{m}$ the principal curvatures at points of $M$ are locally defined $(m-1)$-tuples of scalar values, whose sign depends on the choice of the ''inside'' and ''outside'' of the hypersurface, that is, of the direction of the normal vector. If $M$ is orientable (strictly speaking, coorientable) then the notion can be made global.

\begin{defi}
	The {\bf principal curvatures} $\kappa_1, \kappa_2,\ldots,\kappa_{m-1}$ at the point $w\in M$ of a $m-1$ dimensional hypersurface $M$ in $\mathbb R^{m}$ are the eigenvalues of the second fundamental form $\mathrm{I\!I}(w)$ of $M$ at $w$, that is of the $(m-1)\times(m-1)$ matrix
	 $$A(w):=\begin{pmatrix}
	 	\mathrm{I\!I}(w)(T_1,T_1) &\cdots &  \mathrm{I\!I}(w)(T_1,T_{m-1}) \\
	 	\vdots &\ddots& \vdots\\ 
	 	\mathrm{I\!I}(w)(T_{m-1},T_1) &\cdots & \mathrm{I\!I}(w)(T_{m-1},T_{m-1}) 
	 \end{pmatrix},$$
 where $T_{j},\, j=1,\ldots,m-1$ form an orthonormal  basis of the tangent space  $T_wM$ (orthonormality is with respect to the inner product on $T_wM$, which is the restriction of the standard inner product on $T_w\mathbb R^{m}\cong\mathbb R^{m}$), $\mathrm{I\!I}(w)(X,Y)$ is the bilinear form $\mathrm{I\!I}(w)(X,Y)=-\langle d \nu(w) (X),Y\rangle_{T_wM}$, and $\nu$ is the Gauss map $w\to\nu(w)$ which maps the point $w\in M$ to the unit normal vector to $M$ at $w$ directed ''inside''.
 
  Alternatively,  $\kappa_1, \kappa_2,\ldots,\kappa_{m-1}$ are given as the eigenvalues of the matrix
 $$B(w):=\begin{pmatrix}
 	\mathrm{I}_{11} &\cdots & 	\mathrm{I}_{1(m-1)}\\
 	\vdots &\ddots & \vdots\\
 		\mathrm{I}_{(m-1)1}&\cdots &	\mathrm{I}_{(m-1)(m-1)}
 \end{pmatrix}^{-1}\begin{pmatrix}
 \mathrm{I\!I}_{11} &\cdots & 	\mathrm{I\!I}_{1(m-1)}\\
 \vdots &\ddots & \vdots\\
 \mathrm{I\!I}_{(m-1)1}&\cdots &	\mathrm{I\!I}_{(m-1)(m-1)}
\end{pmatrix},$$
where $ \mathrm{I}_{jk}$ and $\mathrm{I\!I}_{jk}$ are the coefficients of the first and second fundamental forms respectively in a given, not necessarily orthonormal, basis of the tangent space at $w$.

 The {\bf principal directions} are tangent directions of $M$ at $w$ given by the eigenvectors corresponding to $\kappa_1,\ldots,\kappa_{m-1}$. 
\end{defi}

The coefficients of the first fundamental form in the not necessarily orthonormal basis $e_1,\ldots, e_{m-1}$ of $T_wM$ are given by $\mathrm{I}_{jk}=\langle e_j,e_k\rangle_{T_wM}$. This is just the Gram matrix of the basis. So, the first fundamental form can be thought of as the Riemannian metric induced on the hypersurface by the Euclidean metric of $\mathbb R^{m}$. In the same basis the coefficients of the second fundamental form are given by $\mathrm{I\!I}_{jk}=\mathrm{I\!I}(e_j,e_k)$.

The negative of the differential of the Gauss map, that is $-d\nu$, is called the Weingarten map or the shape operator. Note that formally $-d\nu$ sends $T_w M$ to the tangent space of the unit sphere  at $\nu(w)$. It is, however, parallel to, and hence can be identified with $T_wM$. Thus, at $w\in M$ the shape operator can be thought of as a linear transformation $S:T_wM\to T_wM$. The formula $\mathrm{I\!I}(X,Y)=\mathrm{I}(S(X),Y)$ is well-known and sometimes useful.

If the hypersurface is locally parameterized by
$$\mathbb R^{m-1}\supset U\ni (t_1,\ldots,t_{m-1})=t\to \varphi(t)=(\varphi_{1}(t),\ldots,\varphi_{m}(t))\in M\subset \mathbb R^{m}$$
in some local coordinate system $(t_1,\ldots,t_{m-1})$, where we assume that the Jacobian of $\varphi$ is of maximal rank at $\varphi^{-1}(w)$, then 
the Gauss map is 
 $\nu(w)=\nu(\varphi(t))=(n_1(w),\ldots,n_{m}(w))^{T}$ where
$$n_k(w)=\frac{(-1)^{k+m}\det\begin{pmatrix}
		\frac{\partial \varphi_{1}}{\partial t_1} &\cdots & \frac{\partial \varphi_{1}}{\partial t_{m-1}}\\
		\vdots &\ddots& \vdots\\ 
		\frac{\partial \varphi_{k-1}}{\partial t_1} &\cdots & \frac{\partial \varphi_{k-1}}{\partial t_{m-1}}\\[0.5em]
		\frac{\partial \varphi_{k+1}}{\partial t_1} &\cdots & \frac{\partial \varphi_{k+1}}{\partial t_{m-1}}\\
		\vdots &\ddots& \vdots\\
		\frac{\partial \varphi_{m}}{\partial t_1} &\cdots & \frac{\partial \varphi_{m}}{\partial t_{m-1}}
\end{pmatrix}}{\sqrt{{\displaystyle \sum_{j=1}^{m}}\left(\det\begin{pmatrix}
			\frac{\partial \varphi_{1}}{\partial t_1} &\cdots & \frac{\partial \varphi_{1}}{\partial t_{m-1}}\\
			\vdots & \ddots& \vdots\\
			\frac{\partial \varphi_{j-1}}{\partial t_1} &\cdots & \frac{\partial \varphi_{j-1}}{\partial t_{m-1}}\\[0.5em]
			\frac{\partial \varphi_{j+1}}{\partial t_1} &\cdots & \frac{\partial \varphi_{j+1}}{\partial t_{m-1}}\\
			\vdots &\ddots& \vdots\\
			\frac{\partial \varphi_{m}}{\partial t_1} &\cdots & \frac{\partial \varphi_{m}}{\partial t_{m-1}}
		\end{pmatrix}\right)^2}},\quad k=1,\ldots,m.$$
	Thus, computing the vector-valued differential form $d\nu=(d n_1(w),\ldots, dn_{m}(w))^{T}$ and hence $A(w)$ or $B(w)$ is very involved if $m>3$. 

What is important is that the principal curvatures do not depend on the choice of a particular local parameterization.
  If the parameterization is that of a graph of a $C^2$ function $\psi$:  
 $$\mathbb R^{m-1}\supset U\ni (t_1,\ldots,t_{m-1})=t\to \varphi(t)=(t_1,\ldots,t_{m-1},\psi(t_1,\ldots,t_{m-1}))\in M\subset \mathbb R^{m},$$ so that the ''inside'' is the epigraph (or ''above''),  
 the computations simplify significantly. 
 The Gauss map then reads:
  $$\nu(w)=\left(\frac{-\frac{\partial \psi}{\partial t_1}}{\sqrt{1+\Vert \nabla \psi \Vert^2}},\cdots,\frac{-\frac{\partial \psi}{\partial t_{m-1}}}{\sqrt{1+\Vert \nabla \psi \Vert^2}},\frac{1}{\sqrt{1+\Vert \nabla \psi \Vert^2}}\right)^{T}.$$
  (we use the column notation for vectors and ''$T$'' denotes transposition).
  
  The second fundamental form reads:
  \begin{align*}\mathrm{I\!I}(w)(X,Y)&=-\langle d \nu(w) (X),Y\rangle\\&=\sum_{i=1}^{m-1}\sum_{j=1}^{m}\frac{\frac{\partial^2 \psi }{\partial t_{j}\partial t_{i}}}{\sqrt{1+ \Vert \nabla \psi(t) \Vert^2}} X_{j}Y_{i}- \sum_{i=1}^{m-1}\sum_{j=1}^{m}\frac{\frac{\partial\sqrt{1+ \Vert \nabla \psi(t) \Vert^2}}{\partial t_{j}} \frac{\partial \psi }{\partial t_{i}}}{{1+ \Vert \nabla \psi(t) \Vert^2}} X_{j}Y_{i}\\ &-\sum_{j=1}^{m}\frac{\frac{\partial 1 }{\partial t_{j}}}{\sqrt{1+ \Vert \nabla \psi(t) \Vert^2}} X_{j}Y_{m}+ \sum_{j=1}^{m}\frac{\frac{\partial\sqrt{1+ \Vert \nabla \psi(t) \Vert^2}}{\partial t_{j}} }{{1+ \Vert \nabla \psi(t) \Vert^2}} X_{j}Y_{m}.\end{align*}
  The third sum vanishes and the second and fourth add to zero, because
  $$-\sum_{i=1}^{m-1}\sum_{j=1}^{m}\frac{\frac{\partial\sqrt{1+ \Vert \nabla \psi(t) \Vert^2}}{\partial t_{j}} \frac{\partial \psi }{\partial t_{i}}}{{1+ \Vert \nabla \psi(t) \Vert^2}} X_{j}Y_{i} + \sum_{j=1}^{m}\frac{\frac{\partial\sqrt{1+ \Vert \nabla \psi(t) \Vert^2}}{\partial t_{j}} }{{1+ \Vert \nabla \psi(t) \Vert^2}} X_{j}Y_{m}$$$$=\left(\sum_{j=1}^{m} \frac{\frac{\partial\sqrt{1+ \Vert \nabla \psi(t) \Vert^2}}{\partial t_{j}} }{{1+ \Vert \nabla \psi(t) \Vert^2}} X_{j} \right)\left\langle \left(-\frac{\partial \psi }{\partial t_{1}},\ldots,-\frac{\partial \psi }{\partial t_{m-1}},1\right)^{T} ,Y\right\rangle=0,$$
  as $Y$ is a tangent vector. Finally, $$\mathrm{I\!I}(w)(X,Y)=\frac{\mathcal H (\psi(t))(\tilde X,\tilde Y)}{\sqrt{{1+ \Vert \nabla \psi(t) \Vert^2}}},$$ where $\tilde X=(X_1,\ldots X_{m-1})$ for $X=(X_1,\ldots,X_m)$ and the same for $\tilde Y$. Above $\frac{\partial \psi }{\partial t_{m}}$ is understood as zero, $\mathcal H (\psi(t))$ is the real Hessian of $\psi$ at $t=\varphi^{-1}(w)$ and the bilinear form $\mathcal H (\psi(t))(\tilde X,\tilde Y)$ is identical with the expression $\tilde Y^{T}\mathcal H (\psi(t))\tilde X$.
   Observe that in the basis $e_1=\frac{\partial \varphi}{\partial t_1},\cdots,e_{m-1}=\frac{\partial \varphi}{\partial t_{m-1}}$  one has $\tilde e_1=(1,0,\ldots, 0)^{T},\ldots, \tilde e_{m-1}=(0,\ldots,0,1)$, so this is the standard basis of $\mathbb R^{m-1}$. Hence, 
   \begin{align} \begin{pmatrix}
   	\mathrm{I\!I}_{11} &\cdots & 	\mathrm{I\!I}_{1(m-1)}\\
   	\vdots &\ddots & \vdots\\
   	\mathrm{I\!I}_{(m-1)1}&\cdots &	\mathrm{I\!I}_{(m-1)(m-1)}
   \end{pmatrix}=&\begin{pmatrix}
   \frac{\mathcal H (\psi(t))(\tilde e_1,\tilde e_1)}{\sqrt{{1+ \Vert \nabla \psi(t) \Vert^2}}} &\cdots & 	\frac{\mathcal H (\psi(t))(\tilde e_1,\tilde e_{m-1})}{\sqrt{{1+ \Vert \nabla \psi(t) \Vert^2}}}\\
   \vdots &\ddots & \vdots\\
   \frac{\mathcal H (\psi(t))(\tilde e_{m-1},\tilde e_1)}{\sqrt{{1+ \Vert \nabla \psi(t) \Vert^2}}}&\cdots &	\frac{\mathcal H (\psi(t))(\tilde e_{m-1},\tilde e_{m-1})}{\sqrt{{1+ \Vert \nabla \psi(t) \Vert^2}}}
\end{pmatrix}\nonumber\\=&\frac{\mathcal H (\psi(t))}{\sqrt{{1+ \Vert \nabla \psi(t) \Vert^2}}}.\label{finally}\end{align}

  In the same basis $\frac{\partial \varphi}{\partial t_1},\cdots,\frac{\partial \varphi}{\partial t_{m-1}}$ the matrix of the first fundamental form has coefficients $$\mathrm{I}_{ij}=\left\langle \left(\delta_{i1},\ldots,\delta_{i(m-1)} ,\frac{\partial\psi}{\partial t_i}\right)^{T},\left(\delta_{j1},\ldots,\delta_{j(m-1)} ,\frac{\partial\psi}{\partial t_j}\right)^{T}\right\rangle=\delta_{ij}+\frac{\partial \psi}{\partial t_i}\frac{\partial \psi}{\partial t_j}$$ and hence it's inverse satisfies:
  \begin{equation}\label{sherman}\left(Id_{(m-1)\times(m-1)}+\nabla \psi \nabla \psi ^{T}\right)^{-1}=Id_{(m-1)\times(m-1)}-\frac{\nabla \psi \nabla \psi ^{T}}{1+\Vert \nabla \psi\Vert^2},\end{equation}
by the simplest form of the Sherman-Morrison formula for the inverse matrix.
 
If the hypersurface is locally  (in some neighborhood $V$ of $w$) implicitly given as the set of solutions $$M\supset\{ (x_1,\ldots,x_{m})= x\in V\subset \mathbb R^{m} | F(x)=0\}$$ of some equation $F=0$, with $w\in M$ and $\nabla F(w)\neq 0$ then 
the Gauss map is:
$$\nu(w)=-\frac{\nabla F(w)}{\Vert \nabla F(w)\Vert}.$$ 
Here we assume that the ''inside'' of the hypersurface is where $F<0$.
The second fundamental form reads:
\begin{align}\mathrm{I\!I}(w)(X,Y)=&-\langle d \nu(w) (X),Y\rangle=\sum_{i=1}^{m}\sum_{j=1}^{m}\frac{\frac{\partial^2 F }{\partial x_{j}\partial x_{i}}}{\Vert \nabla F(w) \Vert} X_{j}Y_{i}- \sum_{i=1}^{m}\sum_{j=1}^{m}\frac{\frac{\partial\Vert\nabla F\Vert}{\partial x_{j}} \frac{\partial F }{\partial x_{i}}}{\Vert \nabla F(w) \Vert^2} X_{j}Y_{i}\nonumber \\
=&\frac{\mathcal H(F(w))(X,Y)}{\Vert \nabla F(w)\Vert}\label{finally1}\end{align}
because $$\sum_{i=1}^{m}\sum_{j=1}^{m}\frac{\frac{\partial\Vert\nabla F\Vert}{\partial x_{j}} \frac{\partial F }{\partial x_{i}}}{\Vert \nabla F(w) \Vert^2} X_{j}Y_{i}=\left(\sum_{j=1}^{m}\frac{\frac{\partial\Vert\nabla F\Vert}{\partial x_{j}} X_{j}}{\Vert \nabla F(w) \Vert^2}\right)\langle \nabla F(w),Y\rangle=0,$$
	as $Y$ is a tangent vector.
	
	To compute the coefficients of the first fundamental form we assume $\frac{\partial F}{\partial x_m}\neq 0$ and try to solve  $F=0$ for the last variable, that is to find $\psi$ such that $F(x_1,\ldots,x_{m-1},\psi(x_1,\ldots,x_{m-1}))=0$.The implicit function theorem gives $\frac{\partial \psi}{\partial x_j}=-\frac{\frac{\partial F}{\partial x_j}}{ \frac{\partial F}{\partial x_m}}$. Thus, as above, the coefficients of the first fundamental form in the basis $\frac{\partial \varphi}{\partial t_1},\cdots,\frac{\partial \varphi}{\partial t_{m-1}}$ are $$\mathrm{I}_{ij}=\left\langle \left(\delta_{i1},\ldots,\delta_{i(m-1)} ,\frac{\partial\psi}{\partial t_i}\right)^{T},\left(\delta_{j1},\ldots,\delta_{j(m-1)} ,\frac{\partial\psi}{\partial t_j}\right)^{T}\right\rangle=\delta_{ij}+\frac{\partial \psi}{\partial t_i}\frac{\partial \psi}{\partial t_j}=\delta_{ij}+\frac{\frac{\partial F}{\partial x_i}}{ \frac{\partial F}{\partial x_m}}\frac{\frac{\partial F}{\partial x_j}}{ \frac{\partial F}{\partial x_m}}$$ and hence it's inverse satisfies
\begin{align}\left(Id_{(m-1)\times(m-1)}+\nabla \psi \nabla \psi ^{T}\right)^{-1}=&Id_{(m-1)\times(m-1)}-\frac{\nabla \psi \nabla \psi ^{T}}{1+\Vert \nabla \psi\Vert^2}\nonumber\\=&Id_{(m-1)\times(m-1)}-\frac{\left(\frac{\partial F}{\partial x_1},\cdots,\frac{\partial F}{\partial x_{m-1}}\right)\left(\frac{\partial F}{\partial x_1},\cdots,\frac{\partial F}{\partial x_{m-1}}\right)^{T}}{\Vert \nabla F\Vert^2},\label{sherman1}\end{align}
The principal curvatures do not depend on the choice of a particular $F$ which defines the hypersurface locally.

\begin{defi}
One defines the {\bf mean curvature} at the point $w\in M$ as the average $$H=H(w):=\frac{\kappa_1+\kappa_2+\cdots+\kappa_{m-1}}{m-1}.$$
\end{defi}
The definition is extrinsic, meaning that it depends on how $M$ is embedded in $\mathbb R^{m}$ and is not invariant with respect to smooth diffeomorphisms of the ambient space. For example, a scaling by a factor $r$ results in $H$ transforming to $\frac{1}{r}H$. Moreover, 
one has to specify  which side of the hypersurface is outer. 

From linear algebra we know that the sum of the eigenvalues is the trace of a matrix, so
$$H(w)=\frac{\kappa_1+\kappa_2+\cdots+\kappa_{m-1}}{m-1}=\frac{1}{m-1}\tr A(w)=\frac{1}{m-1}\tr B(w).$$
As above, if the hypersurface is locally parameterized by the graph of a $C^2$ function $\psi$   
$$\mathbb R^{m-1}\supset U\ni (t_1,\ldots,t_{m-1})=t\to \varphi(t)=(t_1,\ldots,t_{m-1},\psi(t_1,\ldots,t_{m-1}))\in M\subset \mathbb R^{m}$$
 then by the definition of $B(w)$, \eqref{finally} and \eqref{sherman} we have \begin{align} H(w)=&\frac{1}{m-1}\tr B(w)=\frac{1}{m-1}\tr \left(Id_{(m-1)\times(m-1)}-\frac{\nabla \psi \nabla \psi ^{T}}{1+\Vert \nabla \psi\Vert^2}\right)\frac{\mathcal H (\psi(t))}{\sqrt{{1+ \Vert \nabla \psi(t) \Vert^2}}}\nonumber \\
 	=&\frac{1}{m-1}\frac{{\displaystyle\sum_{i=1}^{m-1}\sum_{j=1}^{m-1}}\left(\delta_{ij}-\frac{\frac{\partial \psi}{\partial t_i}\frac{\partial \psi}{\partial t_j}}{1+\Vert \nabla \psi \Vert^2}\right)\frac{\partial^2\psi}{\partial t_i\partial t_j}}{\sqrt{1+\Vert \nabla \psi \Vert^2}}.  \label{meancurv}\end{align}
 If further the direction of $\nu(w)$ coincides with the $m$-th axis, that is if $\nabla \psi=0$ at $\varphi^{-1}(w)$, then \eqref{meancurv} simplifies to $H(w)=\Delta \psi(t)$.

 If the hypersurface is locally implicitly given as the set of solutions $$M\supset\{ (x_1,\ldots,x_{m})= x\in V\subset \mathbb R^{m} | F(x)=0\}$$ of some equation $F=0$, with $w\in M$ and $\nabla F(w)\neq 0$ we can simplify the calculations by fixing an orthonormal basis $T_1,\ldots,T_{m-1}$ of the tangent space $T_wM$. Then $T_1,\ldots,T_{m-1}, \nu(w)$ is an orthonormal basis of $\mathbb R^{m}$ and hence by the definition of $A(w)$ and \eqref{finally1} we have
 \begin{align}H(w)=&\frac{1}{m-1}\tr A(w)=\frac{1}{m-1}\sum_{j=1}^{m-1}\mathrm{I\!I}(w)(T_j,T_j)=\frac{1}{m-1}\sum_{j=1}^{m-1}\frac{\mathcal H(F(w))(T_j,T_j)}{\Vert \nabla F(w)\Vert}\nonumber\\=&\frac{1}{m-1}\left(\tr\frac{\mathcal H(F(w))}{\Vert \nabla F(w)\Vert} -\frac{\mathcal H(F(w))(\nu(w),\nu(w))}{\Vert \nabla F(w)\Vert} \right)\nonumber\\=&\frac{1}{m-1}\left( \frac{\Delta F(w)}{\Vert \nabla F(w)\Vert}-\frac{\frac{-\nabla F(w)}{\Vert \nabla F(w)\Vert}^{T} \mathcal H(F(w))\frac{-\nabla F(w)}{\Vert \nabla F(w)\Vert}}{\Vert \nabla F(w)\Vert}\right)\nonumber\\
 =&\frac{1}{m-1}\frac{\nabla F(w)^{T} ( (\Delta F(w))Id_{m\times m}-\mathcal H (F(w)))\nabla F(w)}{\Vert \nabla F(w)\Vert^3}  \label{implicitform}.\end{align}
 
 We used the fact that if $Q$ is the matrix of the orthogonal transformation sending the $j$-th vector $e_j$ of the standard basis of $\mathbb R^{m}$ to the $j$-th vector of the basis $T_1,\ldots,T_{m-1}, \nu(w)$ then $Q^{T}=Q^{-1}$ by orthogonality and
 $$\left(\sum_{j=1}^{m-1} T_j^{T}\mathcal H(F(w)) T_j\right)+\nu(w)^{T}\mathcal H(F(w)) \nu(w)=\sum_{j=1}^{m}e_j^{T}Q^{T}\mathcal H(F(w))Qe_j$$$$=\sum_{j=1}^{m}e_j^{T}Q^{-1}\mathcal H(F(w))Qe_j=\tr Q^{-1}\mathcal H(F(w))Q =\tr \mathcal H(F(w)),$$
 as the traces of similar matrices are equal.
 
 By noticing that if one has a graph parameterization $$(t_1,\ldots,t_{m-1})=t\to \varphi(t)=(t_1,\ldots,t_{m-1},\psi(t_1,\ldots,t_{m-1}))$$ then $F$ can be chosen as $F(x_1,\ldots,x_m)=\psi(x_1,\ldots,x_{m-1})-x_{m}$ one can obtain \eqref{implicitform} directly from \eqref{meancurv} and vice versa.

The notions of principal curvatures and mean curvature can be generalized to higher codimensional submanifolds of $\mathbb R^{m}$ and to submanifolds of Riemannian manifolds. For more on these matters see \cite{D}, which is one of the few places where the above formulas are derived explicitly.
	\section{Proof of the main result and of Theorem \ref{sh}}
	The ''if'' part of Oka's lemma (Theorem \ref{oka}) is well known and does not need separate treatment. It follows that $-\log(d_{\pa\Om_j}(z))$ is a continuous plurisubharmonic exhaustion function, the existence of which guarantees the holomorphic convexity. 
	The proof of the ''only if part'', as well as the considerations on subharmonicity properties, will be divided into several steps:
	
	{\bf Step 1: Reduction to the case of a bounded domain with smooth boundary.} 
	
	As this is completely standard we shall be brief. Fix a domain of holomorphy $\Om\subsetneq\C$. Recall that this implies that $\Om$ is holomorphically convex, which in turn implies that there is a smooth strictly plurisubharmonic exhaustion function $\psi$ on $\Om$ (see \cite{K}). By Sard's theorem there is a sequence $\lbrace t_j\rbrace_{j=1}^\infty\subset\mathbb R$, $t_j\nearrow\infty$ such that 
the domains $\Om_j:=\lbrace z\in\Om\ |\ \psi(z)<t_j\rbrace$ are bounded and their boundary is $C^{\infty}$ smooth. As $\psi$ is plurisubharmonic, the domains $\Om_j$ are also Levi pseudoconvex. Of course, they are even strictly pseudoconvex but we shall not use this fact later on. 

It now suffices to prove that $-\log(d_{\pa\Om_j}(z))$ is plurisubharmonic in $\Om_j$ for each $j$ as $-\log(d_{\pa\Om}(z))$ will then be the pointwise decreasing limit of the (plurisubharmonic) functions $-\log(d_{\pa\Om_j}(z))$ and will be hence plurisubharmonic. We fix $j$ in Step 2 and for notational brevity we suppress this indice.

The subharmonic counterpart of the above reasoning is not so widely known. It is a theorem of Greene and Wu (see \cite{GW}) that any connected non-compact Riemannian manifold allows a smooth strongly subharmonic exhaustion function. Just take the manifold to be any fixed domain $\Omega\subsetneq\mathbb R^{n}$ and the metric to be the Euclidean one. There is no completeness requirement. The rest is as above, just read subharmonic instead of plurisubharmonic.

\begin{rem}\label{shsublevel} In the subharmonic case we will not use the fact that a domain is a sublevel set of a subharmonic function, but just that it has a smooth boundary. Moreover, the former domains do not seem to exhibit clear distinctive geometric features. For example the hyperboloid $\left\{\frac {x^2}{a^2}+ \frac{y^2}{a^2}-\frac{ z^2}{c^2}= d\right\}\subset \mathbb R^3, \frac{2}{a^2}>\frac{1}{c^2}$ has either positive or negative Gaussian curvature, depending on the choice of $d$, and if $d>0$ the mean curvature is positive in some regions and negative in other.  
	\end{rem}  

{\bf Step 2: Smooth analysis on a smoothly bounded domain $\Omega$.}

We shall prove the following claim (compare with Theorem \ref{smooth}):

{\bf Claim:} Let $\Om$ be a smoothly bounded Levi pseudoconvex domain in $\C$. If $-\log(d_{\pa\Om}(z))$ is smooth near $p\in\Om$ then it is plurisubharmonic near $p$ in the sense that the Levi form is semi-positive definite on the whole $\C$.

Note that the claim implies that $-\log(d_{\pa\Om}(z))$ is plurisubharmonic in $\Om\setminus\Sigma$.

Essentially, the claim follows from the reasoning in \cite{HM} once one realizes that the analysis there holds in $\Om\setminus\Sigma$ and not only in a small collar around $\pa\Om$.	We provide an alternative proof for the sake of completeness.

From Lemma \ref{GiTr} we know the full real Hessian of $d_{\pa\Om}$ 
$$  \frac{\pa^2 d_{\pa\Om}}{\pa t_i\pa t_j}(0',t_{2n})=\begin{pmatrix}
	{ \frac{-\kappa_{i}}{1-t_{2n}\kappa_{i}}\delta_{ij}}&\begin{matrix}0 \\ \vdots \end{matrix} \\ 
	
	\begin{matrix}0 & \cdots \end{matrix} &0	
\end{pmatrix}_{i,j=1}^{2n}$$
and hence the problem reduces to a computation. The main issue is that the real coordinates $(t_1,\cdots,t_{2n})$
need not cohere well with the ambient complex coordinates.

Fix a point  $p$ in $\Omega\setminus\Sigma$ so that there is a unique $q\in\pa\Om$ such that
$d_{\pa\Om}(p)=\Vert p-q\Vert $. Changing the coordinates near $q$ as in Lemma \ref{GiTr} we identify $q$ with the coordinate origin. In these coordinates we have $p=(0',d_{\pa\Om}(p))$.

Let $J$ denote the (standard) complex structure operator computed in the basis $\frac{\pa}{\pa t_j},\ j=1,\cdots,2n$. Then, given $N=\sum_{j=1}^{n}n_j\frac{\pa}{\pa t_{2j}}$\ ($n_j\in\mathbb R$), the vector $Z=N-iJ(N)$ is a complex vector of type $(1,0)$ with respect to $J$. The following formula for any $C^2$ function $u$ is well-known (see \cite{HM}):
\begin{equation}\label{1j}
\mathcal L(u(p))(Z,Z)=\frac14(\mathcal H(u(p))(N,N)+\mathcal H(u(p))(J(N),J(N))),
\end{equation}
where $\mathcal H$ is the real $2n$-dimensional Hessian of $u$.

Specializing to $-\log(d_{\pa\Om})$ we compute at $p$
\begin{align}\label{spec}
	\mathcal H(-\log(d_{\pa\Om})(p))(X,X)=&-\frac{\mathcal H(d_{\pa\Om}(p))(X,X)}{d_{\pa\Om}}+\frac{|\langle \nabla d_{\pa\Om},X\rangle|^2}{(d_{\pa\Om})^2}\\
	  =&X^T\begin{pmatrix}
\frac{\kappa_{i}}{d_{\pa\Om}(p)(1-d_{\pa\Om}(p)\kappa_{i})}\delta_{ij}&\begin{matrix}0\\\vdots\\0\end{matrix}\\
\begin{matrix}0&&\cdots&& 0\end{matrix}&\frac{1}{(d_{\pa{\Om}}(p))^2}
\end{pmatrix}X.\nonumber\end{align}

In order to exploit the Levi pseudoconvexity of $\pa\Om$ we have to specify the complex tangent directions at $q$. To this end define the vector $X:=-J\left(\frac{\pa}{\pa t_{2n}}\right)$ (so that $J(X)=\frac{\pa}{\pa t_{2n}}$ is the inner normal at $q$). Fix any vector $W$ which is orthogonal to both $X$ and $J(X)$. Then both $W$ and $J(W)$ belong to $T_q\pa\Om$ and $V:=W-iJ(W)$ is a complex tangent vector of type $(1,0)$. Such vectors span the complex tangent space at $q$ and together with $Z=X-iJ(X)$ they span the whole $\C$.

Any complex vector of type $(1,0)$ can then be written as $V+\alpha Z$ with $W$ and $X$ as above, $\alpha\in \mathbb C$. Multiplying the vector  by $\overline{\alpha}$, which does not affect the sign of the Levi form in direction $V+\alpha Z$ neither the way $V$ is obtained, we can assume that the coefficient in front of $Z$ and hence in front of $X$ (still called $\alpha$) is real. In conclusion, it remains to check the positivity of
\begin{align}\label{hesplushes}
A:=&\mathcal H(-\log(d_{\pa\Om})(p))(W+\alpha X,W+\alpha X)\\
&+\mathcal H(-\log(d_{\pa\Om})(p))(J(W)+\alpha J(X),J(W)+\alpha J(X)),\nonumber\end{align}
where $W=\sum_{s=1}^{2n-1}w_s\frac{\pa}{\pa t_s}$, $J(W)=\sum_{s=1}^{2n-1}\omega_s\frac{\pa}{\pa t_s}$ and $X=\sum_{s=1}^{2n-1}\chi_s\frac{\pa}{\pa t_s}$, $J(X)=\frac{\pa}{\pa t_{2n}}$.

Recalling now (\ref{spec}), the expression (\ref{hesplushes}) becomes
\begin{equation*}
A=\sum_{s=1}^{2n-1}\frac{(w_s^2+\omega_s^2+\alpha^2\chi_s^2+2\alpha w_s\chi_s)\kappa_s}{d_{\pa\Om}(p)(1-d_{\pa\Om}(p)\kappa_{s})}+\frac{\alpha^2}{(d_{\pa\Om}(p))^2}
\end{equation*}
(note that $\mathcal H(d_{\pa\Om}(p))(J(W),J(X))$ vanishes). Exploiting the fact that $\sum_{s=1}^{2n-1}\chi_s^2=1$ and completing the squares we obtain

\begin{align*}
A&=\sum_{s=1}^{2n-1}\frac{(w_s^2+\omega_s^2+\alpha^2\chi_s^2+2\alpha w_s\chi_s)\kappa_s}{d_{\pa\Om}(p)(1-d_{\pa\Om}(p)\kappa_{s})}+\sum_{s=1}^{2n-1}\frac{\alpha^2\chi_s^2}{(d_{\pa\Om}(p))^2}	\\
&=\sum_{s=1}^{2n-1}\frac{(w_s^2+\omega_s^2)\kappa_s}{d_{\pa\Om}(p)(1-d_{\pa\Om}(p)\kappa_{s})}+\sum_{s=1}^{2n-1}\left[\frac{\alpha^2\chi_s^2}{(d_{\pa\Om}(p))^2(1-d_{\pa\Om}(p)\kappa_{s})}+\frac{2\alpha w_s\chi_s\kappa_s}{d_{\pa\Om}(p)(1-d_{\pa\Om}(p)\kappa_{s})}\right]\\
&\geq\sum_{s=1}^{2n-1}\frac{(w_s^2+\omega_s^2)\kappa_s}{d_{\pa\Om}(p)(1-d_{\pa\Om}(p)\kappa_{s})}-\sum_{s=1}^{2n-1}\frac{w_s^2\kappa_s^2}{1-d_{\pa\Om}(p)\kappa_{s}}=\sum_{s=1}^{2n-1}\frac{\omega_s^2\kappa_s}{d_{\pa\Om}(p)(1-d_{\pa\Om}(p)\kappa_{s})}\\
&+\sum_{s=1}^{2n-1}\frac{w_s^2\kappa_s}{d_{\pa\Om}(p)}.
\end{align*}

Observe now the crucial fact: the elementary inequality $\frac{\kappa_s}{1-d_{\pa\Om}(p)\kappa_s}\geq \kappa_s$ holds regardless of the sign of $\kappa_s$ (we learned this trick from \cite{Fo}). Hence, $A$ is further bounded from below by
$$\sum_{s=1}^{2n-1}\frac{(w_s^2+\omega_s^2)\kappa_s}{d_{\pa\Om}(p)}.$$
The latter expression is easily seen to be equal to 
$$\frac{\mathcal H((h(t')-t_{2n})(0',0))(W,W)+\mathcal H((h(t')-t_{2n})(0',0))(J(W),J(W))}{d_{\pa\Om}(p)}$$
$$=4\frac{\mathcal L((h(t')-t_{2n})(0',0))(V,V)}{d_{\pa\Om}(p)}.$$

As $h(t')-t_{2n}$ is a local defining function for $\pa\Om$, the last expression is non-negative from the very definition of Levi pseudoconvexity. This finishes the proof of the Claim.

The subharmonic counterpart is much easier. By the same computation as above, and because the Laplacian is the trace of the Hessian one has:

$$\Delta (-\log(d_{\pa\Om}(x)))=\frac{1}{(d_{\pa\Om}(x))^2}+\sum_{i=1}^{m-1}\frac{\kappa_{i}}{d_{\pa\Om}(x)(1-d_{\pa\Om}(x)\kappa_{i})}$$ 
and
$$\Delta (d_{\pa\Om}(x))^{2-m}=\frac{(m-2)(m-1)}{(d_{\pa\Om}(x))^m}+(m-2)\sum_{i=1}^{m-1}\frac{\kappa_{i}}{(d_{\pa\Om}(x))^{m-1}(1-d_{\pa\Om}(x)\kappa_{i})}.$$

 Recall that $1-d_{\pa\Om}(x)\kappa_{i}>0$ for any $x\in\Omega\setminus\Sigma$ and any $i$ (see also Lemma 2.2 in \cite{LLL}). Observe that the function $x\to \frac{x}{1-dx}$ is convex for $x\in\left(-\infty, \frac{1}{d}\right)$, as its second derivative is $\frac{2d}{(1-dx)^3}> 0$. Thus, by the Jensen inequality for convex functions, one has

$$\Delta (-\log(d_{\pa\Om}(x)))\geq\frac{1}{(d_{\pa\Om}(x))^2}+(m-1)\frac{\frac{\sum_{i=1}^{m-1}\kappa_{i}}{m-1}}{d_{\pa\Om}(x)\left(1-d_{\pa\Om}(x)\frac{\sum_{i=1}^{m-1}\kappa_{i}}{m-1}\right)}$$
$$=\frac{1}{(d_{\pa\Om}(x))^2}+\frac{(m-1)\sum_{i=1}^{m-1}\kappa_{i}}{d_{\pa\Om}(x)\left(m-1-d_{\pa\Om}(x)\sum_{i=1}^{m-1}\kappa_{i}\right)}=\frac{m-1+(m-2)d_{\pa\Om}(x)\sum_{i=1}^{m-1}\kappa_{i}}{(d_{\pa\Om}(x))^2\left(m-1-d_{\pa\Om}(x)\sum_{i=1}^{m-1}\kappa_{i}\right)}$$
and respectively
\begin{align*}\Delta (d_{\pa\Om}(x))^{2-m}&\geq\frac{(m-2)(m-1)}{(d_{\pa\Om}(x))^m}+(m-2)(m-1)\frac{\frac{\sum_{i=1}^{m-1}\kappa_{i}}{m-1}}{(d_{\pa\Om}(x))^{m-1}\left(1-d_{\pa\Om}(x)\frac{\sum_{i=1}^{m-1}\kappa_{i}}{m-1}\right)}\\
&=(m-2)(m-1)\left(\frac{1}{(d_{\pa\Om}(x))^m}+\frac{\sum_{i=1}^{m-1}\kappa_{i}}{(d_{\pa\Om}(x))^{m-1}\left(m-1-d_{\pa\Om}(x)\sum_{i=1}^{m-1}\kappa_{i}\right)}\right)\\
&=\frac{(m-2)(m-1)^2}{(d_{\pa\Om}(x))^m\left(m-1-d_{\pa\Om}(x)\sum_{i=1}^{m-1}\kappa_{i}\right)}>0.\end{align*}

In the former case, the subharmonicity is reduced to the non-negativity of $$m-1+(m-2)d_{\pa\Om}(x)\sum_{i=1}^{m-1}\kappa_{i}=(m-1)(1+(m-2)d_{\pa\Om}(x)H).$$
This is positive near the boundary. If $H$ is non-negative then the subharmonicity follows. If $H<0$ then the subharmonicity holds if  $d_{\pa\Om}(x)\leq\frac{1}{(m-2)|H|}$. Observe also that if $R$ is the inradius of $\Omega$ and $H\geq \frac{-1}{(m-2)R}$ then $1+(m-2)d_{\pa\Om}(x)H\geq 1-\frac{d_{\pa\Om}(x)}{R}>0$.

Finally, to deal with the smoothly bounded pseudoconvex case we assume Theorem \ref{geometric property}, where without loss of generality the smallest eigenvalue is $\kappa_{2n-1}$, and  split the sum
\begin{align*}\Delta (-\log(d_{\pa\Om}(x)))&=\frac{1}{(d_{\pa\Om}(x))^2}+\sum_{i=1}^{2n-1}\frac{\kappa_{i}}{d_{\pa\Om}(x)(1-d_{\pa\Om}(x)\kappa_{i})}\\
	&=\frac{1}{(d_{\pa\Om}(x))^2}+\frac{\kappa_{2n-1}}{d_{\pa\Om}(x)(1-d_{\pa\Om}(x)\kappa_{2n-1})}+\sum_{i=1}^{2n-2}\frac{\kappa_{i}}{d_{\pa\Om}(x)(1-d_{\pa\Om}(x)\kappa_{i})}\\
&\geq \frac{1-d_{\pa\Om}(x)\kappa_{2n-1}+d_{\pa\Om}(x)\kappa_{2n-1}}{d_{\pa\Om}(x)(1-d_{\pa\Om}(x)\kappa_{2n-1})}+
(2n-2)\frac{\frac{\sum_{i=1}^{2n-2}\kappa_{i}}{2n-2}}{d_{\pa\Om}(x)\left(1-d_{\pa\Om}(x)\frac{\sum_{i=1}^{2n-2}\kappa_{i}}{2n-2}\right)}\\
&\geq \frac{1}{d_{\pa\Om}(x)(1-d_{\pa\Om}(x)\kappa_{2n-1})}>0, \end{align*}
by the Jensen inequality.

\begin{rem} We note that the inequality
$$\sum_{i=1}^{m-1}\frac{\kappa_{i}}{1-d_{\pa\Om}(x)\kappa_{i}}\geq (m-1)\frac{\frac{\sum_{i=1}^{m-1}\kappa_{i}}{m-1}}{1-d_{\pa\Om}(x)\frac{\sum_{i=1}^{m-1}\kappa_{i}}{m-1}} $$
is presented as a much deeper fact in \cite{LLL}, see Proposition 2.6 there and Proposition 2.5.3 in \cite{BEL}, for the proof of which one needs the Newton inequality for elementary symmetric polynomials.\end{rem}

{\bf Step 3: Non-smooth analysis on $\Om$.} 

It remains to prove that $-\log(d_{\pa\Om}(\cdot))$ extends (pluri)subharmonically past $\Sigma$ for each $\Om=\Om_j$. As $\pa\Om$ is now smooth it follows from Lemma \ref{CrMa} that $\Sigma$ is of zero Lebesgue measure.  We start with the subharmonic case.

We observe that
$$-\log(d_{\pa\Om}(x))=\sup_{ w\in\pa\Om} -\log\Vert x-w\Vert.$$
The functions $x\to-\log\Vert x-w\Vert$ are not subharmonic for $m>2$, yet the following lower bound for their Laplacian is available 
$$-\Delta\log\Vert x-w\Vert=-\sum_{k=1}^{m}\frac{\partial^2 \frac{1}{2}\log\Vert x-w\Vert^2}{\partial x_k\partial  x_{k}}=-\sum_{k=1}^{m}\frac{\Vert x-w\Vert^2-2(x_k-w_k)^2}{\Vert x-w\Vert^4}= \frac{-(m-2)}{\Vert x-w\Vert^2}.$$
So, $$\overline{\Delta} (-\log\Vert x-w\Vert)=\Delta (-\log\Vert x-w\Vert)\geq \frac{-(m-2)}{d_{\pa\Om}(x)}$$ for any $w\in\pa \Omega$. Then, by Lemma \ref{supf}, the same estimate holds for the regularized pointwise supremum:  $\overline{\Delta} (-\log(d_{\pa\Om}(x)))^{\ast}\geq \frac{-(m-2)}{d_{\pa\Om}(x)}$. As the pointwise supremum is continuous there is no need to take the regularization and  so $\overline{\Delta} (-\log(d_{\pa\Om}(x)))\geq \frac{-(m-2)}{d_{\pa\Om}(x)}>-\infty$ everywhere in $\Omega$. 
 Also, by the subharmonicity of $-\log(d_{\pa\Om}(x))$ on $\Omega\setminus\Sigma$ established in Step 2, one has $\overline{\Delta} (-\log(d_{\pa\Om}(x)))\geq 0$ there, that is, almost everywhere.
  By  the theorem of Privalov, $-\log(d_{\pa\Om}(x))$ is subharmonic on $\Omega$. When we consider $(d_{\pa\Om}(x))^{2-m}$ the proof is the same: $(d_{\pa\Om}(x))^{2-m}=\max_{w\in\pa\Omega}\Vert x-w\Vert^{2-m}$ and $\Delta \Vert x-w\Vert^{2-m}=0$.

  The proof in the subharmonic case is over. We now take the plurisubharmonic one. Since $\Sigma$ is of zero Lebesgue measure, $-\log(d_{\pa\Om}(z))$ is plurisubharmonic on $\Om\setminus\Sigma$ by Step 2, and subharmonic in $\Om$ by Step 3 above, we invoke now Theorem \ref{didi} to directly obtain the claimed plurisubharmonicity of $-\log(d_{\pa\Om}(z))$ on the whole $\Om$.
 
 \section{Proof of Theorem \ref{meanconvex}}
 
 This is now very brief.   Observe that
$\overline\Delta(-d_{\pa\Om}(x))\geq \frac{-(m-1)}{d_{\pa\Om}(x)}$ because  $\Delta(-\Vert x-w\Vert)=\frac{-(m-1)}{\Vert x-w\Vert}>\frac{-(m-1)}{d_{\pa\Om}(x)}$ for $w\in\pa\Omega$ and the reasoning in Step 3 above can be repeated. By assumption $-d_{\pa\Om}(x)$ is subharmonic on $\Omega\setminus \Sigma$. Hence, $-d_{\pa\Om}(x)$ is subharmonic in $\Omega$ by the Privalov theorem. Thus, $-d_{\pa\Om}(x)$ being subharmonic on $\Omega\setminus \Sigma$ and on the whole $\Omega$ is the same, provided that $\Sigma$ is a Lebesgue null set.
 
\section{Proof of Theorem \ref{geometric property}}
Essentially, the proof boils down to linear algebra.

First, we have to specify what will be called ''the real part of a complex hyperplane''. Note that given a complex vector space with a fixed complex structure $J$ there is no natural or canonical choice of a splitting into the real and imaginary parts - it is up to a choice of an antilinear involution $v\to c(v)=\overline{v}$. However, once  such a splitting is fixed once can induce a compatible splitting on any complex hyperplane. To see this, let $V$ be a complex vector space with $\dim_{\mathbb C}V=n$, $J$ be the complex structure, $Re V$ and $Im V =J(Re V)$ be the splitting, so that $V=Re V\oplus Im V$  and $\dim_{\mathbb R} Re V= \dim_{\mathbb R}Im V=n$. Let $H$ be a complex hyperplane in $V$. From
$$n+2n-2=\dim_{\mathbb R}Re V+ \dim_{\mathbb R} H=\dim_{\mathbb R} Re V\cap H + \dim_{\mathbb R} Re V + H\leq \dim_{\mathbb R} Re V\cap H +2n  $$ 
we see that $\dim_{\mathbb R} Re V\cap H$ is either $n, n-1$ or $n-2$. The first possibility is ruled out as there is a linear bijection  $ Re V\cap H\ni v\to J(v)\in Im V\cap H$, meaning that $ Im V\cap H= Im V$ and hence $H= V$, a contradiction. So there is a $v\in  Re V\setminus H$. But now $J(v)\in Im V\setminus H$ and hence $J(v)\not \in Re V + H$. So, $\dim_{\mathbb R} Re V + H\leq 2n-1$ and $\dim_{\mathbb R} Re V\cap H\neq n-2$ either. Thus, $\dim_{\mathbb R} Re V\cap H=n-1$ and we define this subspace to be the real part of $H$.

As $\Omega$ is a pseudoconvex domain in $\C$ we first split (say in the standard way) $\C= Re\, \C\oplus Im\,\C$ and for any tangent complex hyperplane we define its real part accordingly.
 
The setting is now as in Step 2 above. Let $U\subset \mathbb R^{2n-1}=\mathop{span} \left\{\frac{\partial}{\partial t_1},\cdots,\frac{\partial}{\partial t_{2n-1}} \right\}$ be the real real part of the complex tangent space at $q$. This subspace is perpendicular to $X=-J\left(\frac{\pa}{\pa t_{2n}}\right)$  and for any vector $W\in U$,  $W-iJ(W)$ is a complex vector of type $(1,0)$ which is tangent to $\partial\Omega$ at $q$. Likewise, $J(U)$ is the imaginary part of the complex tangent space at $q$.  As above, $\dim_{\mathbb R} U= \dim_{\mathbb R} J(U)=n-1$ and $\dim_{\mathbb R} U\oplus J(U)=2n-2$ since $J$ is an orthogonal transformation.

We put the matrix $K$ to be $K:=\begin{pmatrix}\kappa_{1}& 0 &\cdots & 0\\
	0&\kappa_{2}&\cdots&0\\
	\vdots&\vdots&\ddots &\vdots\\
	0&0&\cdots&\kappa_{2n-1}
\end{pmatrix}$. 
Note that at the end of the proof of the Claim in Step 2 above we obtained the inequality
$$\sum_{s=1}^{2n-1}\frac{(w_s^2+\omega_s^2)\kappa_s}{d_{\pa\Om}(p)}\geq 0.$$
Disregarding the denominator, the latter is equivalent to (in matrix notation):

\begin{equation}\label{levi1}\begin{pmatrix}W^{T} & J(W)^{T}\end{pmatrix}\begin{pmatrix}K& \bigzero_{(2n-1)\times(2n-1)}\\[1em]
		\bigzero_{(2n-1)\times(2n-1)} & K
	\end{pmatrix}\begin{pmatrix} W\\[1em] J(W)\end{pmatrix}\geq 0\end{equation}
and
\begin{equation}\label{levi2}\begin{pmatrix}J(W)^{T} & W^{T}\end{pmatrix}\begin{pmatrix}K& \bigzero_{(2n-1)\times(2n-1)}\\[1em]
		\bigzero_{(2n-1)\times(2n-1)} & K
	\end{pmatrix}\begin{pmatrix} J(W)\\[1em] W\end{pmatrix}\geq 0\end{equation}
for any $W\in U$, and where by abusing notation $J$ is the restriction of the complex structure to $\{X, JX\}^{\perp}\subset \mathbb R^{2n-1}$.

We use the following standard fact from linear algebra (a real version of the so-called Ky Fan maximum principle): for a symmetric real   $m\times m$ matrix $A$ the sum of the $j\leq m$ greatest eigenvalues of $A$ is equal to $\max_{B}\tr B^{T}AB$ where the maximum is taken over all real matrices $B$ with $m$ rows and $j$ columns such that $B^{T}B$ is the $j\times j$ identity matrix.

If $e_1,\ldots e_{n-1}$ is any orthonormal basis of $U$  it follows from the block structure of the matrix that
$$2\max_{j\in\{1,\ldots, 2n-1\}}\kappa_1+\cdots+\kappa_{j-1}+\hat{\kappa}_j+\kappa_{j+1}+\cdots+\kappa_{2n-1}$$$$\geq \tr B^{T}\begin{pmatrix}K& \bigzero_{(2n-1)\times(2n-1)}\\[1em]
	\bigzero_{(2n-1)\times(2n-1)} & K
\end{pmatrix}B $$
where $$B= \underbrace{ \begin{pmatrix} e_1 &\cdots& e_{n-1}&J(e_1)& \cdots &J(e_{n-1})\\[1em] J(e_1)& \cdots &J(e_{n-1})&e_1 &\cdots& e_{n-1}\end{pmatrix}}_{2n-2 \text{ columns }}\left.\begin{matrix}\\[2.5em]\end{matrix}\right\}{\scriptstyle\scriptsize 4n-2\text{ rows }}  .$$
But from \eqref{levi1}, \eqref{levi2} all the diagonal entries of the product matrix above are non-negative (positive if the Levi form is non-degenerate on the corresponding vector $e_s-iJ(e_s)$). Hence, we get the first part of the result.

Also, the matrix $\begin{pmatrix}K& \bigzero_{(2n-1)\times(2n-1)}\\[1em]
	\bigzero_{(2n-1)\times(2n-1)} & K
\end{pmatrix}$, or rather the real bilinear form represented by it, is semi-positive definite when restricted to the following $2n-2$ dimensional subspace of $\mathbb R^{4n-2}$: $$\mathop{span}\left\{ \begin{pmatrix}
e_1\\[0.5em]J(e_1)
\end{pmatrix},\cdots,\begin{pmatrix}
e_{n-1}\\[0.5em]J(e_{n-1})
\end{pmatrix},\begin{pmatrix}
J(e_{1})\\[0.5em]e_{1}
\end{pmatrix},\cdots,\begin{pmatrix}
J(e_{n-1})\\[0.5em]e_{n-1}
\end{pmatrix} \right\}.$$
	It follows from linear algebra (or from the theory of Krein spaces) that at least $n-1$ of the eigenvalues, that is of the numbers $\kappa_1,\ldots, \kappa_{2n-1}$, are non-negative.
	In the strongly pseudoconvex case one has  positive definiteness instead of semi-positive definiteness above and hence positive eigenvalues.
	
	\begin{rem} Theorem \ref{geometric property} is sharp in the sense that one can find a pseudoconvex domain with exactly $n-1$ non-negative principal curvatures at some point of its boundary. Just take a  close enough approximation of the product of $n$ copies of the planar annulus $B(0,1)\setminus\overline{B(0,r)}$, which can be obtained from a smooth exhaustion. 
	\end{rem}
	\begin{rem}
		In the geometrically ideal situation, when the principal directions of the second fundamental form coincide with the coordinate axes with respect to the standard coordinates of $\C$, we can choose $W=\frac{\partial}{\partial t_{2s+1}},\, s=0,\ldots,n-2$ to obtain
		\begin{equation}\label{kappasum}\begin{cases}
				\kappa_1+\kappa_2&\geq 0\\
				\kappa_3+\kappa_4&\geq 0\\
				&\vdots\\
				\kappa_{2n-3}+\kappa_{2n-2}&\geq 0\\
			\end{cases}.\end{equation}
		The same observation  can be found in \cite{G}, p.24.
	\end{rem}
	
	\bibliographystyle{amsplain}

\end{document}